\documentclass[oneside,reqno,12pt]{amsart}
 
\usepackage{amsmath, amsfonts, amsthm, amssymb}
\usepackage[english]{babel}
   
\usepackage[utf8]{inputenc}
\usepackage{lmodern}
\usepackage[x11names]{xcolor} 
\usepackage{graphicx}
\usepackage{float}
\usepackage{braket} 
\usepackage[stable, bottom]{footmisc}
\usepackage{wrapfig}
\usepackage{sidecap} 
 \usepackage[margin=2.75cm]{geometry}
 \usepackage{scalerel}
 \usepackage{fontawesome}
  \newtheorem{innercustomthm}{Theorem}
  \newenvironment{customthm}[1]
  {\renewcommand\theinnercustomthm{#1}\innercustomthm}
  {\endinnercustomthm}
 
\usepackage{mathtools} 
\usepackage[shortlabels]{enumitem}
\usepackage{caption}
 
\usepackage{ textcomp } 
\usepackage{relsize}
 \usepackage{csquotes}
  
  \usepackage{palatino}
 \usepackage{lettrine}
 \usepackage{erewhon}
 \usepackage{GoudyIn}
 
 \usepackage{hyperref}
 \hypersetup{
	colorlinks=true,
	linkcolor=Red4, 
	citecolor=blue
 }

\newcommand{\scd}{\mathcal{S}(\C^d)}

\newcommand{\C}{{\mathbb C}}
\newcommand{\N}{{\mathbb N}}

\newcommand{\aand}{\textrm{\ \  and \  }}

\newcommand{\tr}{\operatorname{tr}}

\newcommand{\spann}{\operatorname{span}}

\definecolor{darkred} {cmyk}{0 , 1 ,1 , 0.4 }

 \usepackage{tikz}
 \usetikzlibrary{arrows}

 \usetikzlibrary{decorations.pathmorphing}
 \usetikzlibrary{positioning}

\usepackage{epstopdf}
\newtheorem{theorem}{Theorem}
\newtheorem*{theorem*}{Theorem}
\newtheorem{corollary}[theorem]{Corollary}
\newtheorem*{corollary*}{Corollary}
\newtheorem{lemma}[theorem]{Lemma}
\newtheorem*{lemma*}{Lemma}

\newtheorem{proposition}[theorem]{Proposition}
\newtheorem{definition}[theorem]{Definition}
\newtheorem*{definition*}{Definition}
\theoremstyle{definition}

\newtheorem*{remark*}{Remark}
\newtheorem{obs}[theorem]{Observation}
\newtheorem*{obs*}{Observation}
\newtheorem{ex}[theorem]{Example}
\newtheorem*{ex*}{Example}

\usepackage[backend=bibtex,firstinits=true,citestyle=numeric,bibstyle=numeric,maxbibnames =100,maxcitenames=5,isbn=false]{biblatex}
\renewbibmacro{in:}{}
\DeclareFieldFormat{pages}{#1}

\NewDocumentCommand{\overarrow}{O{=} O{\downarrow} m}{%
	\overset{\makebox[0pt]{\begin{tabular}{@{}c@{}}#3\\[0pt]\ensuremath{#2}\end{tabular}}}{#1}
}
\NewDocumentCommand{\underarrow}{O{=} O{\downarrow} m}{%
	\underset{\makebox[0pt]{\begin{tabular}{@{}c@{}}\ensuremath{#2}\\[0pt]#3\end{tabular}}}{#1}
}
\DeclareMathSymbol{\shortminus}{\mathbin}{AMSa}{"39}

\begin{document}

  \title{      
  Ergodicity of Kusuoka measures on quantum trajectories
}
   \author{
   	Anna Szczepanek
   }
   \address{Institute of Mathematics, Jagiellonian University, \L ojasiewicza 6, 30-348 Krak\'{o}w, Poland}
     \email{ anna.szczepanek@uj.edu.pl}

  \begin{abstract}
In 1989 Kusuoka started the study of probability measures on the shift space that are defined with the help of  products of matrices. In particular, he derived a sufficient condition for the ergodicity of such measures, which have since been referred to as Kusuoka measures. 
We observe that repeated measurements  of a unitarily evolving quantum system generate a Kusuoka measure on the space of sequences of measurement outcomes. We show that if the measurement consists of scaled projections, then Kusuoka's sufficient ergodicity condition can be significantly simplified. We then prove that this condition is also necessary for ergodicity 
if the measurement consists of
uniformly scaled rank-$1$ projections  (i.e., it is a rank-$1$ POVM), or of exactly two projections, one of which is rank-$1$. For the latter class of measurements we also show that the Kusuoka measure is reversible in the sense that every string of outcomes  has the same probability of being emitted by the system as its reverse.

 	\vspace{1.5mm}
  	
  	\noindent
  	\textsc{Keywords}: Kusuoka measures, ergodicity, symbolic dynamics, unitary matrices, quantum information

 	\vspace{1.5mm}
  	
  	\noindent
  	\textsc{MSC2020}:  37A25, 37B10,  81P45

 	\vspace{-2mm}
  	
  \end{abstract}
   \maketitle 
 
   \section{Introduction \& Preliminaries}
 
 \thispagestyle{empty}
  
 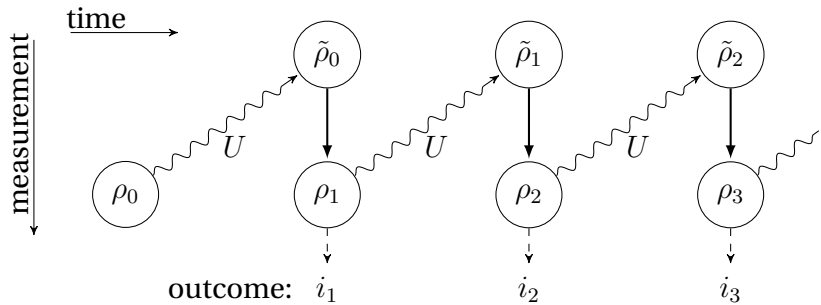
\begin{figure}[b]
 	\vspace{-3mm}
 	\captionsetup{width=0.8\linewidth}
 	\begin{center}
 		\scalebox{0.975}{		\begin{tikzpicture}
 			[>=stealth',shorten >=1pt,auto,node distance=0.995cm,
 			bubbles/.style={circle,font=\sffamily\normalsize,  minimum size=0.9cm}]
 			Bottom
 			\node[bubbles] (st33) at (-1.25, 0) { };
 			\node[bubbles,draw] (st0) at (1.5, 0) {$\tilde{\rho}_0$};
 			\node[bubbles,draw] (st1) at (4.25, 0) {$\tilde{\rho}_1$};
 			\node[bubbles,draw] (st2) at (7, 0) {$\tilde{\rho}_2$};
 			\node[bubbles] (st3) at (8.6, -0.75) {};
 			
 			\node[bubbles,draw] (p33) [below=of st33] {${\rho}_0$};
 			\node[bubbles,draw] (p0) [below=of st0] {${\rho}_1$};
 			\node[bubbles,draw] (p1) [below=of st1] {${\rho}_2$};
 			\node[bubbles,draw] (p2) [below=of st2] {${\rho}_3$};
 			
 			Top
 			\node (w0) at (1.5, -3.2)  {$i_1$};
 			\node (w1) at (4.25, -3.2) {$i_2$};
 			\node (w2) at (7, -3.2) {$i_3$};
 			
 			\path[->,every node/.style={font=\sffamily\normalsize}]	
 			(p33) edge[decorate, decoration={snake, segment length=3mm, amplitude=0.75mm,post length=2mm}] node[below]  {$\ \ U$} (st0)
 			(p0) edge[decorate, decoration={snake, segment length=3mm, amplitude=0.75mm,post length=2mm}] node[below]  {$\ \  U$} (st1)
 			(p1)  edge[decorate, decoration={snake, segment length=3mm, amplitude=0.75mm,post length=2mm}] node[below]  {$\ \ U$} (st2);
 			
 			\path[every node/.style={font=\sffamily\normalsize}]	
 			(p2)  edge[decorate, decoration={snake, segment length=3mm, amplitude=0.75mm}] node[below]  {} (st3);
 			
 			\path[->,dashed,every node/.style={font=\sffamily\normalsize}]	
 			(p0) edge  (w0)
 			(p1) edge  (w1)
 			(p2) edge  (w2);
 			
 			Arrows
 			\draw [-latex,thick](st0) -- (p0);
 			\draw [-latex,thick](st1) -- (p1);
 			\draw [-latex,thick](st2) -- (p2);

 			\draw[->] (-2,0.3) -- (-0.5,0.3);
 
			\node (time) at (-1.65, 0.5) {\normalsize{time}};
 			\node (res) at (0.4, -3.2) {\normalsize{outcome:$\ \ \ \ \ $}};
 			\node[rotate=90]  (res) at (-2.7, -0.95) {\normalsize{measurement}};
 			\draw[->] (-2.5 ,0.2) -- (-2.5,-2.5);
 			\end{tikzpicture}}	
 	\end{center}
  	\vspace{-2mm} 
		\captionsetup{width=0.7\linewidth} 
 	\caption{Repeatedly measured quantum system that between each two consecutive  measurements evolves in accordance to a unitary operator  $\,U$.  State dynamics  $(\rho_0, \rho_1, \ldots)$ is Markovian, outcome dynamics  $(i_1, i_2, \ldots, )$ need not be  Markovian.}
 	\label{fig1} 
 \end{figure}
  
Consider  successive (isochronous) measurements  on a $d$-dimensional \mbox{($d \geq 2$)} quan\-tum-me\-cha\-ni\-cal system that between two
subsequent measurements undergoes deterministic time evolution governed by a unitary operator $U$ (see Fig.~\ref{fig1}). Such a procedure results in the system emitting a sequence of measurement outcomes from $I_k:=\{1, \ldots, k\}$, while the joint evolution of the system can be modelled by a \emph{Partial Iterated Function System} (PIFS). 
 
\begin{definition}\textrm{\rm\cite[p. 59]{Slo03}} 
	The triple $\,(X, \mathsf{F}_i, \mathsf{p}_i)_{i \in I_k}$ is called  a \emph{partial iterated function system (PIFS)} on $X$ if 
	$\hspace{1mm} \mathsf{p}_i\colon \hspace{-0.5mm} X \hspace{-0.5mm}\to [0,1]$,   $\,\sum_{j \in I_k}\mathsf{p}_j=1$, and  $\,\mathsf{F}_i\colon  \{x\hspace{-0.25mm} \in \hspace{-0.25mm}X \hspace{-0.5mm}\colon  \mathsf{p}_i(x)> 0\} \to X$, where $\hspace{0.5mm}i \hspace{-0.25mm}\in \hspace{-0.25mm}I_k$. 
	\end{definition}
  
 \noindent
 Under the action of a PIFS, a given initial state $x \in\hspace{-0.25mm} X$ is transformed into a new state  $\mathsf{F}_i(x)$ with (place-dependent) probability $\mathsf{p}_i(x)$ and the symbol 
$i$ 
corresponding to this evolution is emitted, $i \in I_k$. 
The repeated action of a PIFS generates a Markov chain on $X$ and 
yields sequences of symbols from $I_k$, which can be modelled by a hidden Markov chain. The 
probability and evolution functions related  to these strings are defined inductively in the following natural way. 
Let $n \in \mathbb{N}$, $\iota := (i_1, \ldots, i_n) \in I_k^{n}$ and $j \in I_k$. For $n=1$ both~$\mathsf{p}_\iota$~and~$\mathsf{F}_\iota$ are given. The probability of the system 
outputting 
 $\iota j:= (i_1, \ldots, i_n, j) \in I_k^{n+1}$ is defined  as  
\begin{equation} 
	\label{genprob} \mathsf{p}_{\iota j}(x):=
	\left\lbrace\begin{array}{ll}
		\mathsf{p}_j(\mathsf{F}_{\iota}(x))\mathsf{p}_{\iota}(x) & \textrm{ if }\  \mathsf{p}_{\iota}(x) > 0 \\[0.15em]
		0 & \textrm{ if }\ \mathsf{p}_{\iota}(x) = 0	\end{array}\right.
\end{equation}
and the corresponding evolution map is defined as 
$ \mathsf{F}_{\iota j}(x):=\mathsf{F}_j(\mathsf{F}_{\iota} (x))$ if  $\mathsf{p}_\iota(x) >0$.
Obviously, we have the total probability formula 
\begin{equation}
 	\vspace{0.5mm}
	 \label{totalprob}\mathsf{p}_{{\iota}}(x)=
	\sum_{j \in I_k} \mathsf{p}_{\iota j}(x).
\end{equation}

\begin{remark*}
	The notion of a PIFS generalizes, slightly but significantly, that of an Iterated Function System (IFS) with place-dependent probabilities (see, e.g., \cite{Barnsley1988, peigne1993iterated}) by allowing the evolution map $\mathsf{F}_i$ to remain undefined on the states that have zero probability of being subject to the action of $\mathsf{F}_i$, $i \in I_k$. Such a generalization is necessary in considering quantum measurements, because the evolution associated with a given measurement  outcome 
	cannot be defined on the states with zero probability of producing this outcome, see \eqref{instr}.
\end{remark*}

From this point forward we restrict our attention to PIFSs acting on the set of  quantum states  
 $\:\mathcal{S}(\C^d) \hspace{-0.25mm}:=\hspace{-0.25mm}  \{\rho\hspace{-0.25mm} \in\hspace{-0.25mm} \mathcal{L}(\C^d)\colon  \ \rho \geq 0,\  \tr\rho = 1\}$, where $\mathcal{L}(\C^d)$ denotes the space of (bounded) linear maps on $\,\C^d$. 
The Markov chain generated by such a PIFS on $\hspace{0.25mm}\scd$ corresponds to the so-called  \emph{discrete quantum trajectories}, see, e.g.,
\cite{AttPel11,Benetal17,Kum06,Lim10,MaaKum06}, and 
the sequences of emitted symbols, interpreted as measurement outcomes, form what we can call \emph{coarse-grained quantum trajectories}.
 The study of symbolic dynamics generated by quantum dynamical systems   goes back to \cite{BecGra92, QuantumChaos}, see also 
	\cite{Hartle2, Hartle1}.  
In this paper we employ PIFSs to model repeated measurements performed on unitarily evolving quantum systems and focus on the probability measures that such systems induce on the shift space.

\smallskip

Let us recall the basic mathematical framework of quantum mechanics. A \emph{measurement} of a $d$-dimensional quantum system with $k\hspace{-0.5mm} \in\hspace{-0.5mm} \N$ possible outcomes is given 
	by a \emph{positive operator-valued measure} (\emph{POVM}), i.e., a set of positive semi-definite (non-zero)   operators $\Pi_{1}, \ldots, \Pi_k$ 
	on $\mathbb{C}^{d}$ that sum to the identity, i.e., 
	\begin{equation} 
 	\vspace{-0.5mm}
	\label{sumid}\sum_{j\in I_k}\Pi_{j}=\mathbb{I}.\end{equation}
 
\noindent
 We distinguish  {two} special classes of measurements: 
\begin{itemize}[ leftmargin=*]\itemsep=1mm 
	\item  $\Pi$ is called   a \emph{projection valued measure} (\emph{PVM}) or a \emph{\mbox{L\"{u}ders--von Neumann} measurement$\,$} \cite{Luders} if $\,\Pi_i$ is a projection for every $i\in I_k$. We then have $k \leq d\,$ and the projections constituting $\Pi$ are necessarily orthogonal as self-adjoint projections on a Hilbert space. 
	Moreover, they are mutually orthogonal, i.e.,   $\Pi_i\Pi_j=0$ for $i,j \in I_k$,  $i \neq j$ \cite[p. 46]{halmos1957}.
	
	\item    $\Pi$ is   called a 
	(\emph{normalised}) \emph{rank-$1$ POVM} if   $\hspace{0.25mm}\operatorname{tr}\Pi
	_{1} \hspace{-0.5mm} = \hspace{-0.5mm}\operatorname{tr}\Pi
	_{2}\hspace{-0.5mm} = \hspace{-0.5mm} \ldots \hspace{-0.5mm} = \hspace{-0.5mm} \operatorname{tr}\Pi
	_{k}$ and $\operatorname{rank}\hspace{-0.25mm} \Pi_{i}\hspace{-0.5mm}=\hspace{-0.5mm}1$ for every $i\in I_k$. Then, necessarily,  	$\operatorname{tr}\Pi
_{i}\hspace{-0.375mm} = \hspace{-0.375mm} \frac dk$ for every  $i\in\hspace{-0.375mm} I_k$ and  $k \geq d$. It   follows   that there exist 
 unit vectors $\,{\varphi_{1}}, \ldots,  {\varphi_{k}} \in \mathbb{C}^{d}$   associated with $\Pi$ via 
	$\Pi_{i}= 
\frac dk    \rho_{i}
	$, where $\rho_{i}$ is an orthogonal projection on $\spann\{\varphi_i\}$, $i \in I_k$.

\end{itemize}  

 \smallskip

If the state of the system before the measurement  
is  $\rho\in\mathcal{S}(\mathbb{C}^{d})  $, then the \emph{Born rule} \label{BORNRULE} dictates that the probability
  of obtaining the $i$-th outcome ($i\! \in \!  I_k$) is given by 
$\operatorname{tr}( \Pi_{i} \rho)
$ \cite{Born1926}. \linebreak
The measurement process generically alters the state
of the system, but the POVM alone is not sufficient to determine the
post-measurement  state. This can be done by defining a
\textsl{measurement instrument} (in the sense of Davies and Lewis
\cite{DavLew70}) compatible with $\Pi$, see also \cite{compendium}, \cite[Ch. 10]{busch2016}, \cite[Ch. 5]{HeiZim11}. 
We consider   the   \textsl{generalised L\"{u}ders instruments}, disturbing the initial state in the minimal way, where the input-output state transformation reads 
\vspace{-1.5mm}
\begin{equation} 
	\vspace{2mm}
\label{instr}
\mathcal{S}(\mathbb{C}^{d})\ni \rho \longmapsto \frac{\sqrt{\smash[b]{\Pi_{i}}}\rho \sqrt{\smash[b]{\Pi_{i}}}}{\tr(\Pi_i\rho)}\in\mathcal{S}(\mathbb{C}^{d}),
\end{equation}
provided that the measurement yielded the result $i \in I_k$ 
\cite[p. 404]{DecGra07}, see also  \cite{Barnum2, Barnum}.

\smallskip

Fix a POVM $\Pi=\{\Pi_1, \ldots, \Pi_k\}$ and $U\! \in \mathcal{U}(\C^d)$, where $\,\mathcal{U}(\C^d)$ stands for the set of unitary operators on $\,\C^d$. In what follows we define the  PIFS corresponding to a quantum system that evolves in accordance to $U$ and is repeatedly measured with  $\Pi$. 
 Recall that the deterministic time evolution of a quantum system  is said to be governed  by   $U$ if it is given by the  \emph{unitary channel} acting as
 \vspace{-1mm}
\begin{equation}
\label{channel}
\mathcal{S}(\mathbb{C}^{d})\ni \rho\longmapsto U\rho\,U^{\ast
}\in\mathcal{S}(\mathbb{C}^{d}).
\end{equation}
Taking into account the Born rule,  for an input state  $\rho \in \mathcal{S}(\C^d)$  we define  the probability of obtaining the   outcome $i \in I_k$ as
\vspace{-1mm}
\begin{equation*}
\label{pifsp}
\mathsf{p}_i(\rho):= \tr(\Pi_iU\rho\, U^*).
\end{equation*} The evolution map $\mathsf{F}_i$ is defined as the composition of the unitary channel \eqref{channel} with the state transformation due to $\Pi$ described in \eqref{instr}, i.e.,  
\vspace{1.5mm}
\begin{equation*}
\vspace{0.75mm}	
\label{pifsF}\mathsf{F}_i(\rho):= \frac{\sqrt{\smash[b]{\Pi_{i}}}\,U\rho\, U^*\sqrt{\smash[b]{\Pi_{i}}}}{\tr(\Pi_iU\rho\, U^*)},
\end{equation*}
provided that $\mathsf{p}_i(\rho)>0$. Clearly,  
$(\mathcal{S}(\mathbb{C}^{d}), \mathsf{F}_i, \mathsf{p}_i)_{i \in I_k}$ is a PIFS. 

Next, for $\rho \in \scd$ we put
$\Lambda_i(\rho) := \sqrt{\smash[b]{\Pi_{i}}}U\rho \,  U^*\sqrt{\smash[b]{\Pi_{i}}}$
and observe that 
$\mathsf{p}_i(\rho)\!=\!\tr(\Lambda_{i}(\rho))$  and $ \mathsf{F}_i(\rho)\!=\!\Lambda_{i}(\rho)/\tr(\Lambda_{i}(\rho))$,  provided that  $\tr(\Lambda_{i}(\rho))\!>\!0$.   It follows that  for any initial state $\rho\hspace{-0.25mm} \in \hspace{-0.25mm}\scd$   the probability of the system generating the  string of measurement outcomes $(i_1, \ldots, i_n) \in \hspace{-0.25mm} I_k^n$, where $n \in \N$, is given by 
$\tr(\Lambda_{i_n}   \cdots  \Lambda_{i_1}(\rho)) 
$, i.e.,  
 $$\mathsf{p}_{(i_1, \ldots, i_n)}(\rho)= \tr\hspace{-0.675mm}\left(\hspace{-0.25mm}\sqrt{\smash[b]{\Pi_{i_n}}}U   \cdots 	\sqrt{\smash[b]{\Pi_{i_1}}}	U \rho \, U^* \sqrt{\smash[b]{\Pi_{i_1}}}   \cdots  U^*	\sqrt{\smash[b]{\Pi_{i_n}}}\,	 \right).$$
 
    In a more general setting, probability measures on the shift space that are defined on cylinder sets with the  help of products of matrices corresponding to respective symbols were first investigated by Kusuoka in \cite{Kusuoka1}. Under the name of \emph{Kusuoka measures} they have been mostly explored in the context of fractal geometry, see, e.g., \cite{
  	Strichartz2,  Bessi, Kusuoka2,  Kigami,	Kusuoka1,  liu2017sobolev,  Strichartz1}.  We~stick to the definition of a Kusuoka measure given by Johansson et al. \cite{Kusuoka2}:

 \begin{definition}\label{defKusu}
 	Let $\,\{A_1, \ldots, A_k\} \subset \mathcal{L}(\C^d)$ be such that $\ \sum_{i=1}^k A_iA_i^*=\mathbb{I}\ $ and 
 	$\ \sum_{i=1}^k A_i^*\rho A_i=\rho\ $ for some positive-definite     operator $\hspace{0.5mm}\rho \in\hspace{-0.25mm} \mathcal{L}(\C^d)$ with $\,\tr \rho =1$.  
 	A probability measure $\,\mathbb{P}_\rho$ on  $\hspace{0.5mm}I_k^{\mathbb{N}}$ with the $\,\sigma$-algebra generated by the family of all cylinder sets is called a  \emph{Kusuoka measure   associated to  $\{A_1, \ldots, A_k\}$} if  
 	\begin{align} \mathbb{P}_\rho(C_{(i_1, \ldots, i_n)})    
 	= \tr(A_{i_n}^*  \cdots   A_{i_1}^* \rho \, A_{i_1}  \cdots   A_{i_n}) 
 	\end{align}  
 	where   
 	   $C_{(i_1, \ldots, i_n)}$ stands for the   cylinder set corresponding to the string $(i_1, \ldots, i_n)\! \in\! I_k^n$,   i.e., \linebreak $C_{(i_1, \ldots, i_n)} :=\{(s_i)_{i=1}^\infty \in I_k^\N \colon s_1=i_1,  \ldots, s_n=i_n\}$, $n \in \mathbb{N}$.
 \end{definition}
\noindent 
Observe that the conditions imposed on the operators $A_1, \ldots, A_k$ assure that $\mathbb{P}_\rho$ is consistent (well-defined) and shift-invariant, i.e.,  for every  $\iota \in I_k^n$, where $n \in \N$, we have
 \begin{equation}\label{prop} \mathbb{P}_\rho(C_\iota)=\sum_{j \in I_k} \mathbb{P}_\rho(C_{\iota j})  \ \ \textrm{ and } \  \ \mathbb{P}_\rho(C_\iota) = \sum_{j \in I_k} \mathbb{P}_\rho(C_{j\iota}).\end{equation}

Let us get back to the  quantum system modelled by the PIFS $(\mathcal{S}(\mathbb{C}^{d}), \mathsf{F}_i, \mathsf{p}_i)_{i \in I_k}$. From now on, we fix   the maximally mixed state $\rho_*:=\mathbb{I}/d\,$ as the  initial state of this system.  
For the cylinder set  $C_\iota$  corresponding to   $\iota=(i_1, \ldots, i_n) \in I_k^n$, $n\! \in\! \N$, we put $\mathbb{P}_*(C_\iota)$ for the probability of the system generating $\hspace{0.5mm}\iota\hspace{0.5mm}$ as the string of measurement outcomes, i.e., 
\begin{equation} \label{ptrace}
\mathbb{P}_*(C_\iota)     :=   \:
\mathsf{p}_{\iota}(\rho_*)    
=  \tr\hspace{-0.675mm}\left(\hspace{-0.25mm}\sqrt{\smash[b]{\Pi_{i_n}}}U   \cdots 	\sqrt{\smash[b]{\Pi_{i_1}}}	U \rho_* \, U^* \sqrt{\smash[b]{\Pi_{i_1}}}   \cdots  U^*	\sqrt{\smash[b]{\Pi_{i_n}}}\,	 \right).\end{equation} 
Note that, denoting the   Hilbert-Schmidt norm by $||\cdot ||_{\textsf{HS}}$, we can rewrite \eqref{ptrace}  as 
\begin{equation*} 
	\mathbb{P}_*(C_\iota)      	=  
	\tfrac 1d \left\| U^*\sqrt{\smash[b]{\Pi_{i_1}}}   \cdots U^*	\sqrt{\smash[b]{\Pi_{i_n}}}		\, \right\|_{\textsf{HS}}^2\!.\end{equation*} 
By the Kolmogorov extension theorem, $\hspace{0.5mm}\mathbb{P}_*\hspace{0.5mm}$ extends to a measure on the space of sequences of measurement outcomes $\hspace{0.5mm}I_k^\N\hspace{0.5mm}$ with  the $\sigma$-algebra generated by all cylinder sets. It follows easily that $\mathbb{P}_*$ is a Kusuoka measure associated to  $\{U^*\sqrt{\smash[b]{\Pi_{1}}}, \ldots, U^*\sqrt{\smash[b]{\Pi_{k}}}\}$ since the normalization condition \eqref{sumid} gives $$
\sum_{j \in I_k} (U^*\sqrt{\smash[b]{\Pi_{j}}}\,)(\sqrt{\smash[b]{\Pi_{j}}}U)=\mathbb{I}  \ \ \textrm{ and } \  \ \sum_{j \in I_k} (\sqrt{\smash[b]{\Pi_{j}}}U)\rho_*(U^*\sqrt{\smash[b]{\Pi_{j}}}\,)=\rho_*\, .$$

 \begin{ex}\label{rank1chain}

\smallskip
 
To illustrate the notions introduced so far, let us discuss in detail the   case of $\Pi$ being a rank-$1$ POVM. 
Recall that for each $\,i \in I_k$ we have $\Pi_i = \tfrac dk \rho_i$, where $\rho_i$ is an orthogonal  projection on the subspace spanned by some unit vector $\,\varphi_{i} \in \mathbb{C}^d$. This implies that the evolution maps are all constant since for every $i \in I_k$ we have $\mathsf{F}_i (\rho ) = \rho_i$ for every $\rho \in \scd$ such that  $\mathsf{p}_i(\rho)> 0$. 
That is, to each measurement outcome there corresponds a single post-measurement state,  so from an outcome we can recover the underlying quantum state. 
In consequence, the Kusuoka measure $\mathbb{P}_*$ is a Markov measure, as we now show. 
	
	Firstly, we establish the Markov chain that arises on the space of quantum states $\scd$. In the first measurement each outcome is equally likely:  
	$$\mathsf{p}_{i}(\rho_*) =  \tr(\Pi_i U \rho_*U^*) =\tfrac 1d \tr(\Pi_i) = \tfrac 1k$$ for every $i \in I_k$.
	Hence, the state space of this Markov chain is $\{\rho_1, \ldots, \rho_k\}$, its initial distribution is uniform and the transition matrix reads $[\mathsf{p}_j(\rho_{i})]_{i, j \in I_k}$. 	
Since 
\begin{equation}
	\label{rank1prob}\mathsf{p}_j(\rho_{i})=\tr(\Pi_j U \rho_i U^*) = \tfrac dk \left|  \braket{ \varphi _{j}, U\varphi_{i}} \right|  ^{2}
\end{equation}
 for $\hspace{0.5mm}i, j \hspace{-0.25mm}\in\hspace{-0.25mm} I_k$, we see that the transition matrix is bistochastic, and so the uniform distribution is stationary.

	The dynamics induced by this system on the space of measurement outcomes $\hspace{0.25mm}I_k$ turns out to be Markovian as well. Actually, it mirrors the Markov chain generated on $\hspace{0.25mm}\scd$, i.e., the sequence of quantum states occupied by the system at consecutive time steps can be reconstructed from the sequence of measurement  outcomes.  To see this, let \mbox{$(i_1, \ldots, i_n) \in\hspace{-0.25mm} I_k^n$}, $n \in \N$. We show that 
	 \vspace{-1mm}
\begin{equation}
	\label{markovmeasure}
  	\mathbb{P}_*(C_{(i_1, \ldots, i_n)})= \mathsf{p}_{(i_1, \ldots, i_n)}(\rho_*)=\mathsf{p}_{i_1}(\rho_*)\mathsf{p}_{i_2}(\rho_{i_1})\mathsf{p}_{i_3}(\rho_{i_2})\cdot\ldots\cdot \mathsf{p}_{i_n}(\rho_{i_{n-1}}).
\end{equation} 		
Indeed, if $\mathsf{p}_{(i_1, \ldots, i_n)}(\rho_*)\hspace{-0.25mm} >\hspace{-0.25mm} 0$, then \eqref{genprob} implies that $\mathsf{p}_{(i_1, \ldots, i_r)}(\rho_*)\hspace{-0.25mm} >\hspace{-0.25mm} 0$ 
for 
each \mbox{$r\hspace{-0.25mm} \in\hspace{-0.25mm} \{1, \ldots, n-1\}$}, so
$\mathsf{F}_{(i_1, \ldots, i_r)}(\rho_*)=\rho_{i_r}$. Thus,  $\mathsf{p}_{i_{r+1}}(\mathsf{F}_{(i_1, \ldots, i_r)}(\rho_*))=\mathsf{p}_{i_{r+1}}(\rho_{i_r})$, and \eqref{markovmeasure} follows from \eqref{genprob}.
If $\mathsf{p}_{(i_1, \ldots, i_{n})}(\rho_*)=0$, then 
there exists $r\hspace{-0.5mm} \in\hspace{-0.5mm} \{1, \ldots, n-1\}$ such that 
$\mathsf{p}_{(i_1, \ldots, i_{r})}(\rho_*)\hspace{-0.25mm} >\hspace{-0.25mm} 0$ and $\mathsf{p}_{(i_1, \ldots, i_{r+1})}(\rho_*)=0$. Using \eqref{genprob} again, we obtain   $\mathsf{p}_{i_{r+1}}(\rho_{i_r}) = 0$, and so  \eqref{markovmeasure} holds in this case as well.

Hence, the measurement outcomes form a Markov chain on $I_k$ with  uniform initial distribution and   with  transition matrix  $\mathsf{Q}=[\mathsf{Q}_{ij}]_{i,j\in I_k}$ such that  $\mathsf{Q}_{ij}=\mathsf{p}_{j}(\rho_{i})$ for $\hspace{0.5mm}i,j \in\hspace{-0.5mm} I_k$, and the Kusuoka measure $\mathbb{P}_*$ is a Markov measure, as claimed.
 \end{ex}

\begin{remark*} 
Mimicking the arguments from Example \ref{rank1chain}, one can easily see that if the operators $A_1, \ldots, A_k\,$ that generate  a Kusuoka measure are all rank-$1$, then this Kusuoka measure is a Markov measure. Note that every rank-$1$ operator can be written as the composition of a (scaled) rank-$1$ projection with a unitary operator, as in the case of rank-$1$ POVMs.  
\end{remark*}

Next, we discuss the ergodicity of Kusuoka measures. 
 
  \begin{definition}
 	We say that $\,\{A_1, \ldots, A_k\} \subset  \mathcal{L}(\C^d)$  is  irreducible if  there does not exist a non-trivial subspace of $\:\C^d$ invariant under $A_i\,$ for every $\,i \in  \{1, \ldots, k\}$.\end{definition}  
 \noindent
Kusuoka showed that the irreducibility of a family of operators guarantees the existence and uniqueness of    $\hspace{0.33mm}\rho\hspace{0.33mm}$ from Definition \ref{defKusu}, thus also the existence and uniqueness of the probability measure associated with these operators  \cite[Thm. 1.2]{Kusuoka1}, see also \cite[Prop. 15]{morris2018ergodic}.  
 Moreover, Kusuoka proved that irreducibility constitutes a sufficient condition for the ergodicity of this measure (\cite[Thm. 2.12]{Kusuoka1}, see also \cite[eq. (5)]{Kusuoka2}): 
 \begin{theorem}\label{kusuoka_suff}
 	If $\,\{A_1, \ldots, A_k\}$ is irreducible, then the associated Kusuoka measure is ergodic. \end{theorem}  
  \noindent
   Actually, when it comes to irreducibility, it does not matter whether one considers the operators $A_i$'s or their adjoints. Namely, from the simple fact that a subspace $V\!\subset\textbf{} \C^d$ is invariant under $A\in\! \mathcal{L}(\C^d)$ if and only if  $\hspace{0.33mm}V^\bot$ is invariant under $A^*$, 
 we quickly deduce the following 
 \begin{obs}\label{adjoints}
 	$\{A_1,  \ldots, A_k\}$ is irreducible if and only if   $\{A_1^*, \ldots, A_k^*\}$
 	is irreducible. 
 \end{obs}
\noindent
In our context and with Observation \ref{adjoints} taken into account,   Theorem~\ref{kusuoka_suff} can be restated as
 	\begin{customthm}{5'}
 	\label{ergodremark}
 		If  
 	$\, \{\sqrt{\smash[b]{\Pi_{1}}}U, \ldots, \sqrt{\smash[b]{\Pi_{k}}}U\}$ is irreducible, then $\ \mathbb{P}_*$ is   ergodic.
    	\end{customthm}
    
The main aim of this paper is to show that Theorem~\ref{ergodremark} can be reversed, i.e., that the irreducibility of $\{\sqrt{\smash[b]{\Pi_{1}}}U, \ldots, \sqrt{\smash[b]{\Pi_{k}}}U\}$ is  a  necessary condition for the ergodicity of $\hspace{0.5mm}\mathbb{P}_*$, in the case of $\hspace{0.5mm}\Pi$ being a~rank-$1$ POVM (Theorem~\ref{rank1coroll}) or a PVM consisting of exactly two projections, the ranks of which are equal to $1$ and $\hspace{0.35mm}d-1$, respectively (Theorem~\ref{pvmergodic}). 
    Since for rank-$1$ POVMs $\hspace{0.5mm}\mathbb{P}_*$  is a Markov measure, the characterization of ergodicity via Kusuoka's condition provides an alternative to the well-known characterization in terms of the irreducibility of the transition matrix of the corresponding Markov chain. 
  	
  	A key step in reversing Theorem~\ref{ergodremark} is the simplification of the irreducibility condition in the case of POVMs consisting of scaled orthogonal projections  (Theorem~\ref{projinvariant}). For \mbox{rank-$1$} POVMs this condition has a particularly straightforward geometric description  (Proposition~\ref{suffnonerg}). 
As a result, we can easily characterize when the Kusuoka measure induced by a unitarily evolving qubit  (two-dimensional quantum system) undergoing repeated measurements described by a rank-$1$ POVM is ergodic (Corollary~\ref{qubitkusuoka}).
  	
 Additionally, for the PVMs consisting of two projections with respective ranks $d-\hspace{-0.15mm}1$ and~$1$ we prove that the Kusuoka measure is \emph{reversible} in the sense that the probability of the system emitting a given string of measurement outcomes is equal to the probability that the reverse string will be produced (Theorem \ref{reversible}), i.e.,   
	$$ 
	\mathbb{P}_*(C_{(i_1, \ldots, i_n)}) = \mathbb{P}_*(C_{(i_n, \ldots, i_1)})$$
 for every $\,(i_1, \ldots, i_n) \in I_k^n$, $n \in \N$.
 
 \section{Results}

 Firstly, we show that Kusuoka's sufficient ergodicity condition, i.e., the irreducibility of
  $\{\sqrt{\smash[b]{\Pi_{1}}}U, \ldots, \sqrt{\smash[b]{\Pi_{k}}}U\}$, can be significantly simplified if the POVM $\Pi$  consists of scaled  projections. Namely, instead of verifying the invariance of a subspace of $\,\C^d$ under the composed operators $\sqrt{\smash[b]{\Pi_{1}}}U, \ldots, \sqrt{\smash[b]{\Pi_{k}}}U$, it suffices to   verify its invariance under $U$ and under the measurement operators $\Pi_{1}, \ldots, \Pi_{k}$.

\begin{theorem}\label{projinvariant}
	Let $\,U\hspace{-0.5mm} \in\hspace{-0.25mm}  \mathcal{U}(\C^d)$ and let $\,\Pi=\{\Pi_1, \ldots, \Pi_k\}$ be a POVM  such that  $\,\Pi_{i}={c_i P_{i}}\,$ for every $i \in I_k$, where   $c_i>0\hspace{0.25mm}$ and $\hspace{0.25mm}P_i \in \mathcal{L}(\C^d)$ is an orthogonal projection (i.e., $P_i=P_i^2=P_i^*$).
	Let $\,W$ be a~non-trivial subspace of $\:\C^d$. Then for every $\,i \in I_k$ we have  
$$\sqrt{\smash[b]{\Pi_{i}}}U(W) \subset W \ \Longleftrightarrow \ \ U(W)=W \textrm{ and }\ \  {\Pi_{i}}(W) \subset W.$$
\end{theorem}

\begin{proof}  
	Note that ${\Pi^2_{i}}={c_i\Pi_{i}}$, and so
	$\sqrt{\smash[b]{\Pi_{i}}}={\Pi_{i}}/{\sqrt{\smash[b]{c_{i}}}}$, which guarantees that the images of any linear subspace of $\,\C^d$ under  $\sqrt{\smash[b]{\Pi_{i}}}$ and ${\Pi_{i}}$ coincide.  
	\begin{enumerate}[leftmargin=9.0mm]
		\itemsep 1mm
		\item	
		[($\Rightarrow$)]
We have ${\Pi_{i}}U(W) \subset W$  for every $i \in  I_k$ since  ${\Pi_{i}}U(W)  =  \sqrt{\smash[b]{\Pi_{i}}}U(W) \subset W$. Let $w \in W$. From   \eqref{sumid} we obtain  $Uw  =  \sum_{i=1}^{k}\Pi_iUw$. It follows that $Uw \in W$, because ${\Pi_{i}}Uw \in W$  for every $i \in  I_k$ and $W$ is a   subspace of $\,\C^d$. Hence,  $U(W) \subset W$, thus also $U(W) = W$  as $U$ is an isometry. Therefore, $ {\Pi_{i}}(W) = {\Pi_{i}}U(W) \subset W$ for every $i \in  I_k$, as desired.

		\item 	[($\Leftarrow$)]
		It suffices to observe  that
		$\sqrt{\smash[b]{\Pi_{i}}}U(W)=\sqrt{\smash[b]{\Pi_{i}}}(W)={\Pi_{i}}(W)  \subset W$, where $i\in I_k$. 
		\qedhere
	\end{enumerate}
\end{proof}

 In the case of rank-$1$ POVMs, which consist of uniformly scaled one-dimensional   projections, Kusuoka's sufficient ergodicity condition can be simplified further. Namely, the invariance of a subspace under   $\Pi_{1}, \ldots, \Pi_{k}$ can be expressed in terms of the vectors associated with $\Pi$ belonging to this subspace or to its orthogonal complement. 
In the two following propositions we let $\Pi=\{\Pi_1, \ldots, \Pi_k\}$ be a rank-$1$ POVM and denote the associated unit vectors by $\varphi_{1}, \ldots, \varphi_k$, i.e., for $i \in I_k$ we have  $\Pi_i=\frac{d}{k}\rho_{i}$, where $\rho_{i}$ is an orthogonal projection on $\spann\{\varphi_i\}$.

\begin{proposition} 
	\label{suffnonerg} 
 Let  $\,W\hspace{-0.25mm} $  be a non-trivial   subspace of $\hspace{0.75mm}\C^d$. 	 Then for every $\,i \in  I_k$ we have 
$${\Pi_{i}}(W) \subset W \ \Longleftrightarrow \ \varphi_i \in W \cup W^\bot.$$
\end{proposition} 
\begin{proof} Fix   $\,i \in I_k\,$ and note that 
	$$\ \Pi_i(W)=\{\braket{\varphi_i,w}\varphi_i \colon w \in W\}=	
	\left\lbrace\begin{array}{ll} \{0\}	  & \textrm{ if }\ \varphi_i \in W^\bot,\\\spann\{\varphi_i\} & \textrm{ if }\  \varphi_i \notin  W^\bot.	\end{array}\right.
	$$ 
	\vspace{-2mm}
	\begin{enumerate}
	[leftmargin=9.0mm]
		\itemsep 1mm	 
		\item[($\Rightarrow$)]  
		If $\varphi_i  \notin W^\bot$, then, by assumption, we have  $\,\spann\{\varphi_i\}  =\Pi_i(W) \subset W$, which in turn implies that $\varphi_i \in W$. We conclude that $\varphi_i \in W \cup W^\bot$, as required.
				
		\item[($\Leftarrow$)]
		If $\varphi_i \in W^\bot$, then $\Pi_i(W)=\{0\}$, so  $\Pi_i(W) \subset W$. If $\varphi_i \in W$, then $\Pi_i(W)=\spann\{\varphi_i\}$; hence, we again obtain  $ \Pi_i(W) \subset W$, which concludes the proof. 
	\qedhere
	\end{enumerate}
\end{proof}

Recall from Example \ref{rank1chain} that if $\Pi$ is a rank-$1$ POVM, then  $\mathbb{P}_*$ is a Markov measure. 
It is well known that the ergodicity of a Markov measure is equivalent to  the irreducibility of the corresponding transition matrix \cite[Thm.   6.2.6]{kitchens1997symbolic}. Hence, if $\hspace{0.25mm}\Pi$ is a rank-$1$ POVM, then $\mathbb{P}_*$ is  ergodic if and only if $[\mathsf{p}_j(\rho_i)]_{i,j\in I_k}$ is irreducible.
In what follows we show that in the case of rank-$1$ POVMs Kusuoka's sufficient ergodicity condition follows from the irreducibility of the transition matrix, and so from the non-ergodicity of $\hspace{0.25mm}\mathbb{P}_*$.

\begin{proposition}\label{rank1ergodic}
  
 Let $\hspace{0.25mm}U\! \!\in \!  \mathcal{U}(\C^d)$. If there exists a non-trivial   subspace $\hspace{0.25mm}W\hspace{-0.5mm}$   of $\hspace{1mm}\C^d$    such that   $\hspace{0.15mm}U(W) = W$ and $\hspace{0.5mm}\varphi_i \hspace{-0.25mm}\in\hspace{-0.25mm} W \hspace{-0.15mm} \cup W^\bot$ for every $\,i\hspace{-0.25mm}\in\hspace{-0.25mm} I_k$, then   $\hspace{0.25mm}[\mathsf{p}_{j}(\rho_{i})]_{i,j\in I_k}$ is  reducible.
\end{proposition} 
\begin{proof}  
		
		Put	$I_W:=\{i \in I_k \colon \varphi_{i} \in W\}$. 
		Note that $1 \leq  \# I_W \leq k-1$, because  \eqref{sumid} implies that   $\{\varphi_i\}_{i\in I_k}$ spans $\,\C^d$, and so neither $\hspace{0.25mm}\{\varphi_i\}_{i \in I_k} \subset W$  nor $\{\varphi_i\}_{i \in I_k} \subset W^\bot$ can hold.
		Recall from~\eqref{rank1prob} that 
		$\mathsf{p}_j(\rho_i) = \frac dk \left|\braket{\varphi_j,U\varphi_i}\right|^2\,$ for $\,i, j \in I_k$.	Let  $\hspace{0.25mm}r \in I_W$ and $s \in I_k \setminus I_{W}$. As  both 
		$W$ and $W^{\bot}$ are invariant under $\hspace{0.25mm}U$, we have $\hspace{0.25mm}\varphi_{r} \in W$ and $\hspace{0.25mm}\varphi_{s} \in W^{\bot}$, as well as  $\hspace{0.25mm}U\varphi_r \in W$ and $\hspace{0.25mm}U\varphi_s \in W^{\bot}$. Hence,  $\mathsf{p}_r(\rho_s)=\mathsf{p}_s(\rho_r)=0$, 
		 so it follows easily that  $\hspace{0.25mm}[\mathsf{p}_{j}(\rho_{i})]_{i,j\in I_k}$ is reducible.
\end{proof}

As a result, for rank-$1$ POVMs Kusuoka's sufficient ergodicity condition is also necessary, i.e., Theorem \ref{ergodremark} can be reversed.

\begin{theorem}\label{rank1coroll} Let  $\hspace{0.75mm}U\hspace{-0.75mm}\in  \mathcal{U}(\C^d)$ and let $\hspace{0.5mm}\Pi =\{\Pi_1, \ldots, \Pi_k\} \hspace{0.25mm}$ be a rank-$1$ POVM. The following conditions are equivalent:
	\vspace{-0.25mm}
	\begin{enumerate}[\rm (i)]
		\itemsep=0.5mm
		\item $\mathbb{P}_*$ is  ergodic,
		\item $\{\sqrt{\smash[b]{\Pi_{1}}}U, \ldots, \sqrt{\smash[b]{\Pi_{k}}}U\}$ is irreducible,
	
	\item there is no non-trivial   subspace $\hspace{0.25mm}W\hspace{-0.75mm}$   of $\hspace{1mm}\C^d$  such that $\hspace{0.25mm}U(W)\hspace{-0.25mm}=\hspace{-0.25mm}W$ and  $\hspace{0.5mm}\varphi_i \hspace{-0.25mm}\in\hspace{-0.25mm} W \hspace{-0.15mm} \cup W^\bot$ \linebreak for every $\,i\hspace{-0.25mm}\in\hspace{-0.25mm} I_k$, where
		$\hspace{0.25mm}\varphi_{i}$ is associated with $\Pi$ via 
		$\hspace{0.375mm}\operatorname{im}\hspace{-0.15mm}\Pi_i = \spann\{\varphi_i\}$, $i \in I_k$,

 		\item the transition matrix $\,[\mathsf{p}_{j}(\rho_{i})]_{i,j\in I_k}$ is irreducible.
			\end{enumerate}
\end{theorem}

\begin{proof}As explained above, 
 \mbox{(i) $\Leftrightarrow$ (iv)} is a classical result \cite[Thm.   6.2.6]{kitchens1997symbolic}, \mbox{(iv)  $\Rightarrow$ (iii)} is the contraposition of Proposition~\ref{rank1ergodic}, \mbox{(iii)  $\Rightarrow$ (ii)} follows from Theorem~\ref{projinvariant} coupled with Proposition~\ref{suffnonerg}, and \mbox{(ii) $\Rightarrow$ (i)} is Theorem~\ref{ergodremark}.
\end{proof}

\noindent
In particular, for qubits (two-dimensional quantum systems) we obtain
	\begin{corollary}\label{qubitkusuoka}
	Let $\,U\! \in \mathcal{U}(\C^2)$ and let $\,\Pi$ be a rank-$1$ POVM. Then   $\,\mathbb{P}_*$    is  non-ergodic   if and only if  
$\:\Pi$ is the PVM corresponding to an eigenbasis of $\:U$.
\end{corollary}

We now move on to consider the other class of measurements, i.e., the PVMs   consisting  of  exactly two projections, of which one has  rank $1$, and so the other has    rank   \mbox{$\,d-1$}.   
If $d>2$, then the latter measurement operator gives rise to a non-constant evolution map. In consequence, there may be infinitely many quantum states corresponding to the~same   measurement outcome, which, in principle, causes the symbolic dynamics to be non-Markovian. We start with a simple example of such a PVM producing a non-ergodic Kusuoka measure.

 \begin{ex}	Let $\hspace{0.25mm}U\hspace{-0.5mm} \in  \mathcal{U}(\C^d)$  and let $\Pi=\{\Pi_1, \Pi_2\}$  be a PVM 	such that  $\Pi_1$ and $\Pi_2$ are projections on  $\,\spann\{e_1,   \ldots, e_{d-1}\}$ and    $\,\spann\{e_{d}\}$, respectively, where $\{e_1, \ldots, e_d\}$ is an orthonormal eigenbasis of $\hspace{0.25mm}U$.  
  	
 		In the first measurement both   outcomes are achievable 
 		and their	probabilities are proportional to the dimensions of the respective subspaces:  		 	
		\vspace{-0.5mm}
 		\begin{equation*} 
 		\vspace{-0.5mm}
 		\mathsf{p}_{1}(\rho_*)  =\tfrac {1}{d} \tr(\Pi_1) =  \tfrac {d-1}{d}
 		\ \ \aand \ \ 
 		\mathsf{p}_{2}(\rho_*)  = \tfrac {1}{d} \tr(\Pi_2) =\tfrac {1}{d}.
	 		\end{equation*}
  		Provided that the outcome `$1$' or `$2$' has been obtained, the post-measurement state reads
		\vspace{-0.5mm}
$$
\vspace{-0.5mm}
 		\mathsf{F}_{1}(\rho_*)  =  \tfrac {1}{d-1}\Pi_1 
 		\ \ \textrm{ or }\ \ 
 		\mathsf{F}_{2}(\rho_*)  =  \Pi_2,
 		$$
 		respectively. 
 Next, the probability of the system emitting $j$, provided that the first measurement yielded~$i$, is equal to	
$\mathsf{p}_{j}(\mathsf{F}_{i}(\rho_*))=\delta_{ij}$, where  $\delta_{ij}$ denotes the Kronecker delta and $i,j \in \{1, 2\}$. 
Indeed, observe that if $\,i \neq j$, then 
$
\tr(\Pi_jU\Pi_iU^*)=\tr(\Pi_j\Pi_i)=0,
$
where  the first equality is due to the fact that $U$ and $\Pi_i$ share the eigenbasis $\{e_1, \ldots, e_d\}$, which implies that  
	$\hspace{0.15mm}U\Pi_iU^*\hspace{-0.15mm} = \Pi_i$, 
and the second follows from the mutual orthogonality of $\Pi_1$ and $\Pi_2$.
 	 		 
 		 Hence, the only possible sequences of measurement outcomes are the constant sequence of $\,1$'s, which is generated with probability $\,\tfrac {d-1}{d}$, and the constant sequence of $\,2$'s, generated with complementary probability $\tfrac 1d$.
 		That is, 
 	  $\mathbb{P}_* =  \tfrac {d-1}{d} \mathbb{P}_1 + \frac 1d \mathbb{P}_2$,  where   
 		$\mathbb{P}_i$ stands for the Dirac delta probability measure on $\{1,2\}^\mathbb{N}$ supported on the constant sequence of $i$'s, where   $i \hspace{-0.5mm}\in\hspace{-0.5mm} \{1, 2\}$.  Obviously, $\mathbb{P}_*$ is not ergodic.
 \end{ex} 
 
In the above example all  but one eigenvector of $\,U$ belong to $\operatorname{im}\hspace{-0.25mm}\Pi_1$, where $\Pi_1$ is assumed to be the projection of rank $\,d-1$. It turns out that the presence of an eigenvector of $\,U$ in $\operatorname{im}\hspace{-0.25mm}\Pi_1$ is equivalent to the non-ergodicity of $\,\mathbb{P}_*$, as we now show. 
Note that Theorem \ref{pvmergodic} is in fact the reverse of Theorem~\ref{ergodremark}. A crucial role in the proof is played by the following result:

  \begin{lemma}{\textrm{ \normalfont\cite[Lemma 1]{LAA}}}\label{lemma}
Let  $\,U\! \in \mathcal{U}(\C^d)$  and let  $z \in \C^d$ be a unit vector.   
 Put  	$\,\sigma(U)$  for the set of eigenvalues of $\:U$ and $P$ for the orthogonal projection on $\,\Theta:=\spann\{z\}^\bot$.
 	Then 
 	$$
 	  \vspace{0.5mm}
 	\lim\limits_{m \to \infty}\tr((P U)^m (P U)^{*m})=\sum_{\lambda \in \sigma(U)}\dim(\Theta \cap \operatorname{ker}(U-\lambda \mathbb{I})).$$  
 \end{lemma}

\begin{theorem}\label{pvmergodic} Let  $\hspace{0.25mm}U \hspace{-0.5mm}\in \mathcal{U}(\C^d)$  and let $\,\Pi=\{\Pi_1, \Pi_2\}$ be a PVM   such that  $\,\operatorname{rank}\Pi_1=d-1$ and    $\,\operatorname{rank}\Pi_2=1$.   Put $z$  for a unit vector that spans $\,\operatorname{im}\hspace{-0.25mm}\Pi_2$ and $\,\Theta:= \operatorname{im}\hspace{-0.25mm}\Pi_1 = \spann\{z\}^\bot$. \linebreak The following conditions are equivalent:
	\vspace{-0.25mm}
 	\begin{enumerate}[\rm (i)] 
 		\itemsep 0.5mm

	\item  $\mathbb{P}_*$  is not ergodic,
 		
 		\item there exists a non-trivial  subspace  of $\,\C^d\hspace{-0.5mm}$ 
 		invariant under $U\!$ and under $\hspace{0.25mm}\Pi_2$ 
 			{\rm(}and thus necessarily also under $\hspace{0.25mm}\Pi_1${\rm)},

 		\item 	$z$ belongs to a non-trivial   subspace of $\:\C^d$  invariant under $U$,
 
 		\item  an eigenvector  of $\:U$ belongs to $\,\Theta$.
   		 
 	\end{enumerate}  
 \end{theorem} 

\pagebreak

 \begin{proof}\phantom{x}

 	\begin{enumerate}[leftmargin=20.5mm]	\itemsep 0.5mm

\item	[(i) $\Rightarrow$ (ii)]  This implication follows from  Theorem~\ref{ergodremark} and Theorem~\ref{projinvariant}.

 		\item	[(ii) $\Rightarrow$ (iii)]
Let $W$ be a non-trivial subspace  of $\,\C^d$ invariant under $\hspace{0.25mm}U$ and under $\Pi_2$. 
Clearly, $W^\bot$ is  non-trivial and invariant under $U$ as well. It follows easily that $z \in W$ or  $z \in W^\bot$. Indeed, if $z \notin W^\bot$, then    $\spann\{z\} = \Pi_2(W) \subset W$, so $z \in W$, as desired. 
 		
 		\item	[(iii) $\Rightarrow$ (iv)]	
 
	Let $V$ be a non-trivial subspace  of $\,\C^d$ invariant under $\hspace{0.25mm}U$ and such that $z \in V$. \linebreak 
	We can choose an orthonormal basis $\mathcal{B}_V$ of $\hspace{0.25mm}V$ consisting of the eigenvectors of~$\hspace{0.25mm}U$. As $V^{\bot}$ is invariant under $U$ as well, we can extend   $\mathcal{B}_V$ to an orthonormal basis $\mathcal{B}$ of $\,\C^d$ consisting of the eigenvectors of $\,U$. Each vector from $\mathcal{B}\setminus \mathcal{B}_V$ is an eigenvector of $\,U$ orthogonal to $V$, thus also to $z$, which means that it lies in $\Theta$.

 		\item	[(iv) $\Rightarrow$ (i)] 	
 		Consider $\mathcal{W}:=\{(s_i)_{i=1}^\infty \in I_2^\N \colon    s_i = 1  \:   \textrm{ for almost all }\: i \in \N\}$. Obviously, $\mathcal{W}$ is invariant under the shift operator.
 		Putting 
 		$$ \hspace{15mm}
 		\mathcal{W}_{n,m}:=\{(s_i)_{i =1}^\infty \in I_2^\N \colon  s_{n+1} =\ldots= s_{n+m}=1\},$$
 		we have   $\hspace{0.25mm}\mathcal{W}=\bigcup_{n=0}^\infty \bigcap_{m=1}^\infty \mathcal{W}_{n,m}$, so 
 		from the continuity of  $\,\mathbb{P}_*$ we  obtain  
 		$$ \hspace{15mm}
 		\mathbb{P}_*(\mathcal{W}) = \lim_{n \to \infty }\lim_{m \to \infty} \mathbb{P}_*( \mathcal{W}_{n,m}).
 		$$
 Fix $m \hspace{-0.25mm}\in \hspace{-0.25mm}\N$ and $n\hspace{-0.25mm} \in\hspace{-0.25mm} \N \cup \{0\}$.  For strings consisting exclusively of 1's we adopt the notation 
  $\textbf{1}^m:=({1}, \ldots, 1) \in {I}^m_2$. Since $\mathcal{W}_{n,m}= 	\bigcup_{\kappa \in I_2^n}  C_{\kappa \mathbf{1}^m}$, we obtain
 		\begin{align*} \hspace{15mm}
 		\mathbb{P}_*(\mathcal{W}_{n,m})
 = 	\sum_{\kappa \in I_2^n}	\mathbb{P}_*( C_{\kappa \mathbf{1}^m})
 		=
 	\mathbb{P}_*(C_{\mathbf{1}^m})
&
 		=  \tr ((\Pi_1U)^m \rho_*  (U^*\Pi_1)^{ m} )
 		\\[-0.75em] 
 		&	
 		=	\tfrac 1d     \tr ((\Pi_1U)^m  (\Pi_1U)^{* m}),
 		\end{align*}
where the second equality follows from  \eqref{prop} and the third  from  \eqref{ptrace}. In consequence,  Lemma~\ref{lemma} gives 
 		\begin{align*}
 \hspace{15mm}		\mathbb{P}_*(\mathcal{W})&= \tfrac 1d   \lim_{m \to \infty}  \tr ((\Pi_1U)^m   (\Pi_1U)^{* m})  
 		= 
 		\tfrac 1d \hspace{-1mm} \sum\limits_{\lambda \in \sigma(U)} \hspace{-1mm} \dim(\Theta \cap  		\operatorname{ker}(U-\lambda \mathbb{I})).
 		\end{align*}
By assumption, there is an eigenvector  of  $\hspace{0.25mm}U$ in  $\Theta$. Denoting the corresponding eigenvalue by $\hspace{0.25mm}\tilde{\lambda}$, we obtain    \mbox{$\dim(\Theta \cap   \operatorname{ker}(U-\tilde{\lambda} \mathbb{I}))\geq 1$}; hence 
	$\mathbb{P}_*(\mathcal{W}) \geq \frac 1d>0$.
It remains to observe that  $\bigoplus_{\lambda \in \sigma(U)}(\Theta \cap \operatorname{ker}(U-\lambda \mathbb{I})) \subset \Theta$, and so
\vspace{-1mm}
	$$ \hspace{15mm}
	\sum_{\lambda \in \sigma(U)} \dim(\Theta \cap  		\operatorname{ker}(U-\lambda \mathbb{I})) \leq \dim\Theta = d-1;$$ hence, $\mathbb{P}_*(\mathcal{W}) \leq \frac {d-1}d<1$, which concludes the proof. 
\qedhere
 	\end{enumerate}
 \end{proof}
  
Finally,  we show that PVMs consisting of two projections with ranks equal to  $d-1$ and~$1$, respectively, lead to Kusuoka measures that are reversible in the sense that any given cylinder set has the same measure as the cylinder set corresponding to the reverse string. In other words, the probability of the system outputting  any given string  of measurement outcomes coincides with the probability of it producing these outcomes in reverse order. 
Before moving on to the proof of this claim, we  note that  reversibility is a stronger property than shift-invariance. Indeed, assume that $\,\mathbb{P}_*(C_{(j_1, \ldots, j_m)}) \hspace{-0.5mm} = \mathbb{P}_*(C_{(j_m, \ldots, j_1)})$ for every  $(j_1, \ldots, j_m) \hspace{-0.5mm}\in\hspace{-0.5mm} {I}_k^m$,   $m\in \N$. Then for every $(i_1, \ldots, i_n) \in {I}_k^n$, $n\in \N$, we obtain  
\vspace{-0.5mm}
\begin{align*}
	\sum_{r \in I_k} \mathbb{P}_*(C_{(r, i_1, \ldots, i_n)})
	= \sum_{r \in I_k} \mathbb{P}_*(C_{(i_n, \ldots, i_1, r)})
	\overset{\eqref{totalprob}}{=}  \mathbb{P}_*(C_{(i_n, \ldots, i_1)})
	=  \mathbb{P}_*(C_{(i_1, \ldots, i_n)}).\end{align*}

\vspace{0.25mm}

\begin{theorem}\label{reversible} Let   $\hspace{0.25mm}U\! \in \mathcal{U}(\C^d)$  and   let  $\,\Pi\hspace{-0.5mm}=\hspace{-0.5mm}\{\Pi_1, \Pi_2\}$ be a PVM such that  $\,\operatorname{rank}\Pi_1=d-1$ and  $\operatorname{rank}\Pi_2=1$.  Then 
	$\:\mathbb{P}_*(C_{(i_1, \ldots, i_n)}) = \mathbb{P}_*(C_{(i_n, \ldots, i_1)})\,$ for every  $\,(i_1, \ldots, i_n)\in I_2^n$, $n\in \N$.
\end{theorem}

\begin{proof} Since $\Pi_2$ is a rank-$1$ projection,  	$\mathsf{F}_{2}(\rho)  =  \Pi_2$ for every $\rho \in \scd$ such that  $\mathsf{p}_2(\rho)> 0$. We begin by proving two simple facts. 
	
	\vspace{1mm}
	
	\noindent
	\underline{Fact 1.} 
	Let  $\rho\hspace{-0.25mm} \in \hspace{-0.25mm}\scd$ and $(j_1, \ldots, j_m) \hspace{-0.25mm}\in \hspace{-0.25mm} I_2^m$, $m \hspace{-0.25mm}\in \hspace{-0.25mm}\N$. Assume that for some $r \hspace{-0.25mm}\in\hspace{-0.25mm} \{1, \ldots m-1\}$ we have  $\,j_r=2\,$. Then 
	\vspace{-1.5mm}
	\begin{equation}\label{obsinv1} 
	\vspace{-0.5mm}
	\mathsf{p}_{(j_1, \ldots, j_m)}(\rho)=\mathsf{p}_{(j_1, \ldots, j_r)}(\rho)\mathsf{p}_{(j_{r+1}, \ldots, j_m)}(\Pi_2).\end{equation}
Indeed, if $\mathsf{p}_{(j_1, \ldots, j_r)}(\rho)>0$, then $\mathsf{F}_{(j_1, \ldots, j_r)}(\rho)=\Pi_2$, and the repeated application of \eqref{genprob}  yields the desired formula. 
Similarly, if $\mathsf{p}_{(j_1, \ldots, j_r)}(\rho)=0$, then \eqref{genprob} implies that $\mathsf{p}_{(j_1, \ldots, j_m)}(\rho)=0$.

	\vspace{1mm}
	
	\noindent
	\underline{Fact 2.} Let  $m\in \N$.	We have
 	\vspace{-2mm}	
	\begin{equation}
	\vspace{-1mm}
	\label{obsinv2}    \mathsf{p}_{(\underbrace{\scriptstyle 1, \ldots, 1}_{m},2)}(\rho_*)=\tfrac 1 d  \mathsf{p}_{(\underbrace{\scriptstyle 1, \ldots, 1}_{m})}(\Pi_2). \end{equation}
Indeed, 
it follows that 
 	\vspace{-1mm}	
  $$
  	\vspace{-1mm}	
  \mathsf{p}_{(2,\underbrace{\scriptstyle 1, \ldots, 1}_{m})}(\rho_*)=\mathsf{p}_{2}(\rho_*) \mathsf{p}_{(\underbrace{\scriptstyle 1, \ldots, 1}_{m})}(\Pi_2)  = \tfrac 1d\mathsf{p}_{(\underbrace{\scriptstyle 1, \ldots, 1}_{m})}(\Pi_2),$$ 
 where we first use \eqref{obsinv1} and then the fact that 
 $\mathsf{p}_{2}(\rho_*)  = \tfrac {1}{d} \tr(\Pi_2) =\tfrac {1}{d}$.  
  It remains to show  that   
	$\mathsf{p}_{(2,\underbrace{\scriptstyle 1, \ldots, 1}_{m})}(\rho_*)=	\mathsf{p}_{(\underbrace{\scriptstyle 1, \ldots, 1}_{m},2)}(\rho_*)$. Since $\Pi_2=\mathbb{I}-\Pi_1$, from       \eqref{ptrace} we have  
	\begin{align*}  
	\mathsf{p}_{(2,\underbrace{\scriptstyle 1, \ldots, 1}_{m})}(\rho_*)& =\tr(\Pi_1
	U \cdots 	\Pi_1	U\Pi_2U \rho_*  U^*\Pi_2U^*\Pi_1	\cdots U^* {\Pi_{1}}  )
	\\[-0.95em]&= \tfrac 1d \tr( {\Pi_1}	U\cdots \Pi_1	U\Pi_2U^*\Pi_1
	\cdots U^*{\Pi_{1}}  ) 	
	\\[0.15em]&= 	\tfrac 1d \tr(
	\Pi_1U\cdots  {\Pi_{1}}U ({\mathbb{I}}-\Pi_1) U^* {\Pi_{1}}
	\cdots U^*{\Pi_{1}}   )
	\\[0.15em]&=   \tr(
	\Pi_1U\cdots  {\Pi_{1}}U \rho_*   U^* {\Pi_{1}}
	\cdots U^*{\Pi_{1}}   ) 
		\\[-1.75em] & \hspace{35mm }-\tr(
	\Pi_1U\cdots  {\Pi_{1}}U  (\overbrace{\Pi_1 U \rho_*  U^* \Pi_1}^{
		\scaleto{\tfrac{1}{d}}{10pt}
		  \Pi_1}) U^* {\Pi_{1}}
	\cdots U^*{\Pi_{1}}   )
	\\[-0.15 em] &	= \mathsf{p}_{(\underbrace{\scriptstyle 1, \ldots, 1}_{m})}(\rho_*) -  \mathsf{p}_{(\underbrace{\scriptstyle 1, \ldots, 1}_{m+1})}(\rho_*)
	\\[-0.6em] & = \mathsf{p}_{(\underbrace{\scriptstyle 1, \ldots, 1}_{m}, 2)}(\rho_*),
	\end{align*}
	where the last equality follows from \eqref{totalprob}. We conclude	that  \eqref{obsinv2} holds.
	
	\vspace{1mm}
	
	Also, we let $\epsilon$ stand for the empty string and define  $\mathsf{p}_{\epsilon}(\rho):=1$ for every $\rho\in \scd$. 
	We have
	$\mathsf{p}_{2}(\rho_*)=  \frac{1}{d}=\frac{1}{d}\mathsf{p}_{\epsilon}(\Pi_2)$,
	and so \eqref{obsinv2} holds  for $m=0$ as well.
   
	Now, we fix $\iota=(i_1, \ldots, i_n)\in {I}_2^n$,  $n\in \N$. If $n=1$ or $\iota$ is a string of identical symbols, then the assertion of the theorem holds trivially. We therefore assume that $n \geq 2$ and that both   symbols `$1$' and `$2$' appear in $\iota$. 
	Let $1 \leq {j_1}<{j_2}< \ldots< {j_s}\leq n$, where $s \in \{1, \ldots, n-1\}$, stand for the positions in $\iota=(i_1, \ldots, i_n)$ occupied by $2$, and denote by $l_r:={j_{r+1}}-{j_r}-1$ the number of times that $1$ appears between the \mbox{$r$-th} and \mbox{$(r+1)$-th} occurrence of $2$, where $r \in \{1,\ldots,  s-1\}$. Moreover, put $l_0:={j_1}-1$ and $l_s:=n-{j_s}$ for the number  of $1$'s that appear before the first and after the last appearance of $2$ in $\iota$, respectively. That is, we have 
	$$
	\vspace{1mm}
	\iota=(i_1, \ldots, i_n)=(\underbrace{ 1, \ldots, 1}_{l_0}, \overarrow[2]{$i_{j_1}$} ,\underbrace{ 1, \ldots, 1}_{l_1}, \overarrow[2]{$i_{j_2}$},   \ldots \ldots , \overarrow[2]{$i_{j_{s-1}}$}, \underbrace{ 1, \ldots, 1}_{l_{s-1}}, \overarrow[2]{$i_{j_s}$}, \underbrace{ 1, \ldots, 1}_{l_s}).$$

\noindent	
	The repeated application of \eqref{obsinv1} yields  the following factorization
	\begin{align*}	
		\vspace{1mm}
		\mathbb{P}_*(C_{(i_1, \ldots, i_n)}) & = \mathsf{p}_{(i_1, \ldots, i_{j_1})}(\rho_*)\:  \mathsf{p}_{(i_{j_1+1}, \ldots, i_{j_2})}(\Pi_2)\cdot \ldots \cdot \mathsf{p}_{(i_{j_{s-1}+1}, \ldots, i_{j_s})}(\Pi_2) \: \mathsf{p}_{(i_{j_s+1}, \ldots, i_{n})}(\Pi_2)\\&=\mathsf{p}_{(\underbrace{\scriptstyle 1, \ldots, 1}_{l_0},2)}(\rho_*)\: \mathsf{p}_{(\underbrace{\scriptstyle 1, \ldots, 1}_{l_1},2)}(\Pi_2)\cdot \ldots \cdot \mathsf{p}_{(\underbrace{\scriptstyle 1, \ldots, 1}_{l_{s-1}},2)}(\Pi_2)\: \mathsf{p}_{(\underbrace{\scriptstyle 1, \ldots, 1}_{l_s})}(\Pi_2),\end{align*}
	while for the reversed string we obtain
	\begin{align*}
	\vspace{1mm}	
	\mathbb{P}_*(C_{(i_n, \ldots, i_1)})&=\mathsf{p}_{(i_n, \ldots, i_{j_s})}(\rho_*)\: \mathsf{p}_{(i_{j_s-1}, \ldots, i_{j_{s-1}})}(\Pi_2)\cdot \ldots \cdot \mathsf{p}_{(i_{j_{2}-1}, \ldots, i_{j_1})}(\Pi_2) \:  \mathsf{p}_{(i_{j_1-1}, \ldots, i_{1})}(\Pi_2)\\&=\mathsf{p}_{(\underbrace{\scriptstyle 1, \ldots, 1}_{l_s},2)}(\rho_*)\: \mathsf{p}_{(\underbrace{\scriptstyle 1, \ldots, 1}_{l_{s-1}},2)}(\Pi_2)\cdot \ldots \cdot \mathsf{p}_{(\underbrace{\scriptstyle 1, \ldots, 1}_{l_{1}},2)}(\Pi_2)\:  \mathsf{p}_{(\underbrace{\scriptstyle 1, \ldots, 1}_{l_0})}(\Pi_2).
	\end{align*}
	Clearly, if $\hspace{0.15mm}\mathsf{p}_{(\underbrace{\scriptstyle 1, \ldots, 1}_{l_{r}},2)}(\Pi_2)\hspace{-0.35mm} = \hspace{-0.35mm} 0\hspace{0.15mm}$ for some $\hspace{0.15mm}r \hspace{-0.5mm}\in \hspace{-0.5mm}\{1, \ldots, s-1\}$, then  $\hspace{0.15mm}\mathbb{P}_*(C_{(i_1, \ldots, i_n)}) \hspace{-0.35mm}
	= \hspace{-0.35mm}\mathbb{P}_*(C_{(i_n, \ldots, i_1)})\hspace{-0.35mm}=\hspace{-0.35mm}0$.
	Otherwise, the desired equality $\mathbb{P}_*(C_{(i_1, \ldots, i_n)}) 
	= \mathbb{P}_*(C_{(i_n, \ldots, i_1)})$ is equivalent to 
	\begin{equation}
		\label{dowinverted}
		\vspace{1mm}
	\mathsf{p}_{(\underbrace{\scriptstyle 1, \ldots, 1}_{l_0},2)}(\rho_*)\, \mathsf{p}_{(\underbrace{\scriptstyle 1, \ldots, 1}_{l_s})}(\Pi_2)=\mathsf{p}_{(\underbrace{\scriptstyle 1, \ldots, 1}_{l_s},2)}(\rho_*) \,\mathsf{p}_{(\underbrace{\scriptstyle 1, \ldots, 1}_{l_0})}(\Pi_2).\end{equation}
	To conclude the proof, it remains to observe that \eqref{obsinv2} implies that both sides of~\eqref{dowinverted} are equal to $\frac 1d  \mathsf{p}_{(\underbrace{\scriptstyle 1, \ldots, 1}_{l_0})}(\Pi_2)\,\mathsf{p}_{(\underbrace{\scriptstyle 1, \ldots, 1}_{l_s})}(\Pi_2)$.
\end{proof}

\section*{Acknowledgments}

The author is grateful to Wojciech S{\l}omczy{\'{n}}ski for numerous  helpful  comments and suggestions that greatly improved the presentation of this paper. Financial support of  the Polish National Science Centre under Project No. 2016/21/D/ST1/02414 is acknowledged.

\setlength{\bibitemsep}{0.3\baselineskip}         

\printbibliography

\end{document}